\documentclass[12pt]{article}

\usepackage{epsfig}
\usepackage{latexsym}
\usepackage{amsmath}
\usepackage{amsfonts}
\usepackage{amssymb}
\usepackage{graphicx}%
\usepackage{algorithmic}
\usepackage{color}

\makeatletter
\newcommand{\rmnum}[1]{\romannumeral #1}
\newcommand{\Rmnum}[1]{\expandafter\@slowromancap\romannumeral #1@}
\makeatother

\makeatletter
\let\@fnsymbol\@arabic
\makeatother

\begin{document}
	
	\pagestyle{myheadings} \markright{\sc  Digraph analogues for the Nine Dragon Tree Conjecture \hfill} 
	
	\thispagestyle{empty}
	
	\newtheorem{theorem}{Theorem}[section]
	\newtheorem{corollary}[theorem]{Corollary}
	\newtheorem{definition}[theorem]{Definition}
	\newtheorem{guess}[theorem]{Conjecture}
	\newtheorem{claim}[theorem]{Claim}
	\newtheorem{problem}[theorem]{Problem}
	\newtheorem{question}[theorem]{Question}
	\newtheorem{lemma}[theorem]{Lemma}
	\newtheorem{proposition}[theorem]{Proposition}
	\newtheorem{fact}[theorem]{Fact}
	\newtheorem{acknowledgement}[theorem]{Acknowledgement}
	\newtheorem{algorithm}[theorem]{Algorithm}
	\newtheorem{axiom}[theorem]{Axiom}
	\newtheorem{case}[theorem]{Case}
	\newtheorem{conclusion}[theorem]{Conclusion}
	\newtheorem{condition}[theorem]{Condition}
	\newtheorem{conjecture}[theorem]{Conjecture}
	\newtheorem{criterion}[theorem]{Criterion}
	\newtheorem{example}[theorem]{Example}
	\newtheorem{exercise}[theorem]{Exercise}
	\newtheorem{notation}[theorem]{Notation}
	\newtheorem{observation}[theorem]{Observation}
	\newtheorem{solution}[theorem]{Solution}
	\newtheorem{summary}[theorem]{Summary}
	
	\newtheorem{thm}[theorem]{Theorem}
	\newtheorem{prop}[theorem]{Proposition}
	\newtheorem{defn}[theorem]{Definition}

	\newtheorem{lem}[theorem]{Lemma}
	\newtheorem{con}[theorem]{Conjecture}
	\newtheorem{cor}[theorem]{Corollary}

	\newenvironment{proof}{\noindent {\bf
			Proof.}}{\rule{3mm}{3mm}\par\medskip}
	\newcommand{\remark}{\medskip\par\noindent {\bf Remark.~~}}
	\newcommand{\pp}{{\it p.}}
	\newcommand{\de}{\em}

	\newcommand{\g}{\mathrm{g}}

	\newcommand{\qf}{Q({\cal F},s)}
	\newcommand{\qff}{Q({\cal F}',s)}
	\newcommand{\qfff}{Q({\cal F}'',s)}
	\newcommand{\f}{{\cal F}}
	\newcommand{\ff}{{\cal F}'}
	\newcommand{\fff}{{\cal F}''}
	\newcommand{\fs}{{\cal F},s}
	\newcommand{\cs}{\chi'_s(G)}
	
	\newcommand{\G}{\Gamma}
	\newcommand{\wrt}{with respect to }
	\newcommand{\mad}{{\rm mad}}
	\newcommand{\col}{{\rm col}}
	\newcommand{\gcol}{{\rm gcol}}
	
	\newcommand*{\ch}{{\rm ch}}
	\newcommand*{\ra}{{\rm ran}}
	\newcommand{\co}{{\rm col}}
	\newcommand{\sco}{{\rm scol}}
	\newcommand{\wc}{{\rm wcol}}
	\newcommand{\dc}{{\rm dcol}}
	\newcommand*{\ar}{{\rm arb}}
	\newcommand*{\ma}{{\rm mad}}
	\newcommand{\di}{{\rm dist}}
	\newcommand{\tw}{{\rm tw}}
	\newcommand{\scol}{{\rm scol}}
	\newcommand{\wcol}{{\rm wcol}}
	\newcommand{\td}{{\rm td}}
	\newcommand{\edp}[2]{#1^{[\natural #2]}}
	\newcommand{\epp}[2]{#1^{\natural #2}}
	\newcommand*{\ind}{{\rm ind}}
	\newcommand{\red}[1]{\textcolor{red}{#1}}
	
	\def\C#1{|#1|}
	\def\E#1{|E(#1)|}
	\def\V#1{|V(#1)|}
	\def\iarb{\Upsilon}
	\def\ipac{\nu}
	\def\nul{\varnothing}

	\newcommand*{\QEDA}{\ensuremath{\blacksquare}}
	\newcommand*{\QEDB}{\hfill\ensuremath{\square}}

\title{\Large\bf  Digraph analogues for the Nine Dragon Tree Conjecture}

\author{
	Hui Gao \footnotemark[1]$~^,$\footnotemark[3]~~~~
	Daqing Yang \footnotemark[2]$~^,$\footnotemark[4]$~^,$\footnotemark[5]	
}
\footnotetext[1]{School of Mathematics, China University of Mining and Technology, Xuzhou, Jiangsu 221116, China.} 
\footnotetext[2]{Department of Mathematics,
	Zhejiang Normal University, Jinhua, Zhejiang 321004, China.}
\footnotetext[3]{E-mail: \small\texttt{gaoh1118@yeah.net}}
\footnotetext[4]{Corresponding author,  grant numbers: 
	NSFC  11871439, U20A2068, 11971437.}
\footnotetext[5]{E-mail: \small\texttt{dyang@zjnu.edu.cn}}


\maketitle

\begin{abstract}
	The fractional arboricity of a digraph $D$, denoted by $\gamma(D)$, is defined as
	$\gamma(D)= \max_{H \subseteq D, |V(H)| >1}  \frac {|A(H)|} {|V(H)|-1}$.  	
	 Frank in [Covering branching, Acta Scientiarum Mathematicarum (Szeged) 41 	(1979), 77-81] proved that a digraph $D$ decomposes into $k$ branchings, if and only if $\Delta^{-}(D) \leq k$ and $\gamma(D) \leq k$. 
	 
	 In this paper, we study digraph analogues for the Nine Dragon Tree Conjecture. We conjecture that, for positive integers $k$ and $d$, if $D$ is a digraph with 
	 $\gamma(D) \leq  k + \frac{d-k}{d+1}$ 
	 and $\Delta^{-}(D) \leq k+1$, then $D$ decomposes into $k + 1$ branchings  $B_{1}, \ldots, B_{k}, B_{k+1}$ with $\Delta^{+}(B_{k+1}) \leq d$. This conjecture, if true, is a refinement of Frank's characterization. A series of acyclic bipartite digraphs is also presented to show the bound of $\gamma(D)$ given in the conjecture is best possible.  
	 	 	 
	 We prove our conjecture for the cases $d \leq k$. As more evidence to support our conjecture, we prove that  if $D$ is a digraph with the maximum average degree $\mad(D)$  $\leq$  $2k + \frac{2(d-k)}{d+1}$ and $\Delta^{-}(D) \leq k+1$, then $D$ decomposes into $k + 1$ pseudo-branchings $C_{1}, \ldots, C_{k}, C_{k+1}$ with $\Delta^{+}(C_{k+1}) \leq d$.   	 
 \end{abstract}



\section{Introduction} 
In this paper, we consider graphs or digraphs which may have multiple edges or arcs but no loops.  
For a graph $G$, a {\em decomposition of $G$} consists of edge-disjoint subgraphs with union $G$;   a {\em decomposition of a digraph $D$} is defined similarly. The \emph{arboricity of $G$} is the minimum number of forests needed to decompose it. The \emph{fractional arboricity of $G$} is defined as
$$\gamma(G):=\max_{H \subseteq G, |V(H)|>1} \frac{|E(H)|}{|V(H)|-1}.$$
This notion was introduced by Payan \cite{Pa-86}; see also \cite{Ca-Gr-Ho-Lai-92,Go-06}.  
For a  digraph $D$, the \emph{fractional arboricity of $D$} is defined to be the fractional arboricity of its underlying graph, 
written as $\gamma(D)$.  
The celebrated Nash-Williams' Theorem states a necessary and sufficient condition for the arboricity of $G$ to be at most $k$:

\begin{theorem}[Nash-Williams' Theorem, \cite{Na-64}] \label{nash-williams}
A graph $G$ can be decomposed into $k$ forests if and only if $\gamma(G) \leq k$.
\end{theorem}

When $k<\gamma(G) < k+1$, by Theorem~\ref{nash-williams}, $G$ decomposes into $k+1$ forests. However, when $\gamma(G)=k+ \epsilon$ is closer to $k $ than to $k+1$, then not only does $G$ decompose into $k+1$ forests, but one of the forests can be small (Refer \cite[Corollary 3.7]{FJLWYZ} for this).
Montassier, Ossona de Mendez, Raspaud, and 
Zhu \cite{Mon-Men-Ras-Zhu-12} gave two definitions of this special forest, called \emph{$\epsilon$-forest}: $(i)$  a forest whose maximum degree is bounded by a function of $\epsilon$, $(ii)$ a forest consisting of small subtrees, i.e. subtrees whose sizes are bounded by a function of $\epsilon$. Based on the two different definitions, they proposed the  NDT (Nine Dragon Tree) Conjecture and Strong  NDT  Conjecture. 
A graph is \emph{$d$-bounded} if its maximum degree is at most $d$. 

\begin{conjecture} [The NDT Conjecture, \cite{Mon-Men-Ras-Zhu-12}] \label{NDT-conj}
For non-negative integers $k$ and $d$, if $G$ is a graph with $ \gamma(G) \leq k+ \frac{d}{k+d+1}$, then $G$ decomposes into $k + 1$ forests with one being $d$-bounded.
\end{conjecture}

\begin{conjecture}[The Strong NDT Conjecture, \cite{Mon-Men-Ras-Zhu-12}] \label{strong-NDT-conj}
For non-negative integers $k$ and $d$, if $G$ is a graph with $ \gamma(G) \leq k+ \frac{d}{k+d+1}$, then $G$ decomposes into $k + 1$ forests, with one of them consisting of subtrees of sizes at most $d$. 
\end{conjecture}

In \cite{Mon-Men-Ras-Zhu-12},  Montassier et al. proved the fractional arboricity bound is sharp and the NDT Conjecture holds when $(k,d)=(1,1)$ or $(1,2)$. Many partial results \cite{Chen-Kim-Ko-West-Zhu-17,Kim-Ko-West-Wu-Zhu-13,Yang-18}  on the NDT Conjecture came until Jiang and Yang \cite{Jiang-Yang-17} proved it. The strong NDT Conjecture is only proved for cases  $(k,d)=(1,2)$ \cite{Kim-Ko-West-Wu-Zhu-13} and $d=1$ \cite{Yang-18}. 

\begin{thm}[The NDT Theorem, \cite{Jiang-Yang-17}]
	The NDT Conjecture is true. 
\end{thm}

This paper will study some directed versions of Conjectures~\ref{NDT-conj} and \ref{strong-NDT-conj}. We shall use the concepts of `arborescence' and `branching', which are well studied in digraphs, and are the directed versions of `tree' and `forest'.

Let $D=(V,A)$ be a digraph. 
A subdigraph of $D$ is  spanning if its vertex set is $V$.  A subdigraph $F$ (it may not be spanning) of $D$ is an {\em $r$-arborescence} if its underlying graph is a tree, $r \in V(D)$ and for any $u \in V(F)$, there is exactly one directed path in $F$ from $r$ to $u$. We say that the vertex $r$ is the \emph{root of the  arborescence $F$}.

A \emph{branching} $B$ in $D$ is a spanning subdigraph each component of which is an arborescence, and the \emph{root set} $R(B)$ of $B$  consists of all roots of its components. Frank \cite{Frank-79} characterized a digraph which decomposes into $k$ branchings.

For $v \in V(D)$, let $\partial_{D}^{-}(v)$ denote   the set of arcs in $D$ with $v$ as their head, $\partial_{D}^{+}(v)$ denote the set of arcs in $D$ with $v$ as their tail, $d_{D}^{-}(v):=|\partial_{D}^{-}(v)|$ and $d_{D}^{+}(v):=|\partial_{D}^{+}(v)|$.  
$\Delta^{-}(D) := \max_{v \in V} d_{D}^{-}(v)$, and $\Delta^{+}(D) := \max_{v \in V} d_{D}^{+}(v)$.

\begin{thm}[\cite{Frank-79}]\label{directed-nash-williams-theorem}
	A digraph $D$ decomposes into $k$ branchings if and only if $\Delta^{-}(D) \leq k$ and $\gamma(D) \leq k$.
\end{thm}

Analogous to the NDT Conjecture \ref{NDT-conj}, we propose the following conjecture as its corresponding directed version. 

\begin{conjecture} \label{directed-NDT-conj}
	For positive integers $k$ and $d$, if $D$ is a digraph with 
	$\gamma(D) \leq  k + \frac{d-k}{d+1}$
	and $\Delta^{-}(D) \leq k+1$, then $D$ decomposes into $k + 1$ branchings  $B_{1}, \ldots, B_{k}, B_{k+1}$ with $\Delta^{+}(B_{k+1}) \leq d$.
\end{conjecture}

In Section $2$, we shall prove  Conjecture \ref{directed-NDT-conj} for cases $d \leq k$.  Note that in Conjecture \ref{directed-NDT-conj}, the condition of  $\Delta^{-}(D) \leq k+1$ is necessary: if $D$ decomposes into $k + 1$ branchings  $B_{1}, \ldots, B_{k+1}$, then $\Delta^{-}(B_{i}) \leq 1$ for $1 \leq i \leq k+1$, thus $\Delta^{-}(D) \leq k+1$.   
In Section \ref{Section-lower-B}, by constructing a series of acyclic bipartite  digraphs, we show the following theorem, which says 
that the bound of  $ \gamma(D)$ given in Conjecture~\ref{directed-NDT-conj} is best possible.  

\begin{thm} \label{fra-arb-bound-tight}
	For positive integers $k, d$ and any $\epsilon > 0$, there exists an acyclic bipartite digraph $D$ such that $ \gamma(D) <  k + \frac{d-k}{d+1} + \epsilon$ and $\Delta^{-}(D) \leq k+1$, but $D$ can not be decomposed into $k + 1$ branchings $B_{1}, \ldots, B_{k}, B_{k+1}$ with $\Delta^{+}(B_{k+1}) \leq d$.
\end{thm}  

It is temping to make some analogous conjecture in digraphs corresponding to the  Strong NDT Conjecture \ref{strong-NDT-conj}.   
Unfortunately, we construct an example next which shows some naively analogous   conjecture is not true. 


Construct a sequence of digraphs whose underlying graphs are trees, $D_{0} \subseteq D_{1} \subseteq \ldots \subseteq D_{n}$:  
$(i)$ $D_{0}$ is a single vertex; 
$(ii)$ for each $i \in \{0, 1, \ldots, n-1\}$, construct $D_{i+1}$ from $D_{i}$ by adding  $(k+1)^{i+1}$ new vertices and $(k+1)^{i+1}$ new arcs such that for  already existing vertices $v \in V(D_{i})$, $d_{D_{i+1}}^{-}(v)=k+1$. 
As an example, Figure~\ref{f1} shows the sequence of digraphs  $D_i$, $i \in \{0, 1, 2, 3\}$, of $k=1$ and $n=3$. 
\begin{figure}[hptb]
	\centering
	\includegraphics[width=15cm]{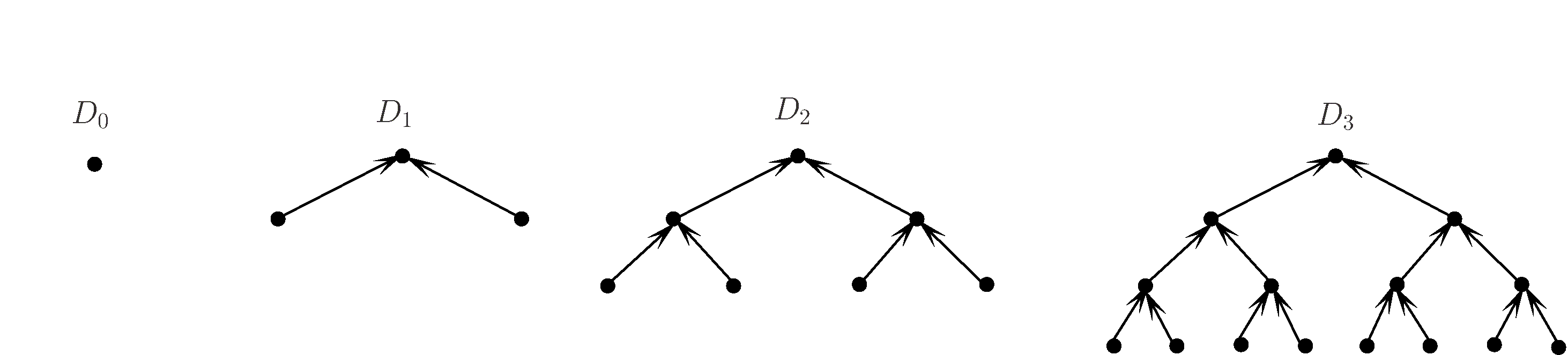}\\
	\caption{$k=1,n=3$} \label{f1}
\end{figure} 


Note that in $D_n$, for any $v \in V(D_{n-1})$, $d_{D_{n}}^{-}(v)=k+1$; for any $ u \in V(D_{n}) \setminus  V(D_{n-1})$, $d_{D_{n}}^{-}(u)=0$, and $u$ is an in-neighbor of some vertex in  $V(D_{n-1})  \setminus  V(D_{n-2})$. This leads to the fact that, up to isomorphism, $D_{n}$ has only one decomposition into $k+1$ branchings, each containing a directed path of length $n$.  

Since $\Delta^{-}(D_{n})=k+1$ and $\gamma(D_{n})=1$,   
this example shows that for $\Delta^{-}(D) \leq k+1$, the fractional arboricity  $1 + \epsilon$ (where $\epsilon$ is any small positive value)  
can not guarantee that a digraph decomposes into $k+1$ branchings with one consisting of arborescences of bounded sizes. So it seems that it is still a puzzle to solve even for a conjecture along this direction.   

A slightly different sparseness condition places a bound on the average vertex degree in all subgraphs. The \emph{maximum average degree of a graph $G$}, denoted by $\mad(G)$, is defined as
$$ \mad(G):= \max_{H \subseteq G}\frac{2|E(H)|}{|V(H)|}.$$
For a digraph $D$, let $\mad(D):= \mad(G[D])$, where $G[D]$ is the underlying graph of $D$.

A graph is a \emph{pseudo-forest} if there exists an orientation $\sigma$ of $G$ such that $\Delta^{-}(G^{\sigma}) \leq 1$,  
 where $G^{\sigma}$ is the oriented graph of $G$ under $\sigma$. 
 Hakimi \cite{Ha-65} characterized a graph which decomposes into $k$ pseudo-forests:
\begin{thm} [\cite{Ha-65}] \label{Hakimi} 
A graph $G$ decomposes into $k$ pseudo-forests if and only if  $\mad(G)$ $\leq$ $ 2k$.
\end{thm}

Fan, Li, Song and Yang \cite{Fan-Li-Song-Yang-15} proved the pseudo-forest analogue of the NDT Theorem. Later, Grout and Moore \cite{Gr-Mo-19} proved the pseudo-forest analogue of the Strong NDT Conjecture.

\begin{theorem}[\cite{Fan-Li-Song-Yang-15}]\label{pseudo-forest-NDT}
For non-negative integers $k$ and $d$, if $G$ is a graph with $ \frac{1}{2}\mad(G)$  $\leq$  $k+ \frac{d}{k+d+1}$, then $G$ decomposes into $k + 1$  pseudo-forests with one being $d$-bounded.
\end{theorem}

\begin{theorem} [\cite{Gr-Mo-19}]\label{directed-strong-NDT}
For non-negative integers $k$ and $d$, if $G$ is a graph with $ \frac{1}{2}\mad(G)$  $\leq k+ \frac{d}{k+d+1}$, then $G$ decomposes into $k + 1$ pseudo-forests $F_{1}, \ldots, F_{k+1}$ such that each component of $F_{k+1} $ has at most $d$ edges.
\end{theorem}




A \emph{pseudo-branching} $C$ in a digraph $D$ is a spanning subdigraph with $\Delta^{-}(C) \leq 1$.

The same example above in the analysis for an analogue in digraphs corresponding to the  Strong NDT Conjecture \ref{strong-NDT-conj}  
says we need some more (unknown) ingredients to have an analogue in digraphs corresponding to Theorem \ref{directed-strong-NDT}.   
Here we give a digraph analogue to Theorem 1.9, stated below:

\begin{thm}\label{main-pseudo-branching-dec}
For positive integers $k$ and $d$, if $D$ is a digraph with $\frac{1}{2}\mad(D)$   $\leq$  $k + \frac{d-k}{d+1}$ 
and $\Delta^{-}(D) \leq k+1$, then $D$ decomposes into $k + 1$ pseudo-branchings $C_{1}, \ldots, C_{k}, C_{k+1}$ with $\Delta^{+}(C_{k+1}) \leq d$.
\end{thm} 

Theorem \ref{main-pseudo-branching-dec} serves as one more piece of  evidence to support  Conjecture \ref{directed-NDT-conj}.  
We shall prove Theorem \ref{main-pseudo-branching-dec} in Section 3.  The bound on  $ \mad(D)$  in Theorem \ref{main-pseudo-branching-dec}  
is best possible, this shall be explained in Section \ref{Section-lower-B}.  

We need some notation for our proofs. 
Suppose $D=(V,A)$ is a digraph.  
For $X \subseteq V$, let $N_{D}^{-}(X) = \{y \in V: y \notin X, \mbox{ and } \exists x \in X, \mbox{ such that } \overrightarrow{yx} \in A \}$,
$N_{D}^{+}(X) = \{z \in V: z \notin X, \mbox{ and } \exists x \in X, \mbox{ such that } \overrightarrow{xz} \in A \}$,  
$D[X]$ denote the subdigraph of $D$ induced by $X$. When $X=\{v\}$, 
we write $N_{D}^{-}(v)$ for $N_{D}^{-}(\{v \})$, and name it the in-neighbors of $v$; $N_{D}^{+}(v)$ for $N_{D}^{+}(\{v \})$, and name it the out-neighbors of $v$. Sometimes, for simplicity, we use $A[X]$ to denote the arc set of $D[X]$.  

For $A_0 \subseteq A$, we define $D-A_0 := (V,A\setminus A_0)$. When  $A_{0}=\{a_{0}\}$, write $D-a_{0}$ for $D-\{a_{0}\}$. For an arc $a_{1} \notin A$ with its ends in $V$, we define $D+a_1 := (V,A \cup \{a_1\})$. 
Let $\Omega$ be a finite set.  For a  function $f:\Omega \rightarrow \mathbb{N}$,  define $\widetilde{f}: 2^{\Omega} \rightarrow \mathbb{N}$ as $\widetilde{f}(X)=\sum_{x \in X}f(x)$, where $X \subseteq \Omega$.


%

\section{The branching analogue for the NDT Theorem}  

In this section, we prove the cases  $d \leq k$ of Conjecture \ref{directed-NDT-conj}.  Our proof relies on the following proposition.   

\begin{proposition} \label{B_{k+1}-exist}
For digraph $D=(V,A)$, suppose $\{S,T \}$ is a partition of $V$, $f: V \rightarrow \mathbb{N}$ is a non-negative function. If for any $\emptyset \neq X \subseteq T $, 
\begin{equation} \label{submodular-function} 
 \widetilde{f}(N_{D}^{-}(X)) \geq |X|, 
\end{equation}
then there exists a branching $B$ in $D$ such that 
\begin{itemize}
\item[(i)] for any $v \in T$, $d_{B}^{-}(v)=1$;
\item[(ii)] for any $v \in V$, $d_{B}^{+}(v)\leq f(v) $.
\end{itemize}
\end{proposition}

\begin{proof}
Initialize the branching $B$ with  vertex set $V$ and empty arc set. 
If  $T=\emptyset$, then $B$ is the branching as demanded.  
Suppose $T \neq \emptyset$. 


In the proof, we use three dynamic variables $S^{*}$, $T^{*}$ and $f^{*}$, their values will be updated  in the process. 
Initially, let  $S^{*}:=S$, $T^{*}:=T$ and $f^{*}:=f$. 
We shall maintain the following three properties at any time of our proof:  
\begin{itemize}
\item[(a)] for any $v \in T \setminus T^*$, $d_{B}^{-}(v)=1$;
\item[(b)] for any $v \in S^*$, $ d_{B}^{+}(v) + f^*(v) \leq  f(v)$;
\item[(c)] for any $v \in T^*$, $d_{B}^{+}(v)=d_{B}^{-}(v)=0$. 
\end{itemize}
Note that initially Properties $(a)$, $(b)$ and $(c)$ are satisfied.   
By Inequality (\ref{submodular-function}),  $\widetilde{f}(N_{D}^{-}(T))$ $=$   $\widetilde{f^*}(N_{D}^{-}(T^*)) \geq |T^*| = |T| >0$;  
since $\{S,T \}$ is a partition of $V$,  $N_{D}^{-}(T) \subseteq S $; 
it follows that there exists $s_{0} \in N_{D}^{-}(T^*) \subseteq S^* $ such that $ f^*(s_{0}) >0$.

In the proof, we shall construct the demanded branching $B$ by adding \lq suitable\rq~ arcs one by one.  To find a suitable arc, we begin with an  $s_{0} \in N_{D}^{-}(T^*) \subseteq S^* $ such that $ f^*(s_{0}) >0$.  Then we look for a \lq suitable\rq~ $t_{0} \in T^*$ such that $s_{0}t_{0} \in A $.  
Then we update $B:=B+ \overrightarrow{s_{0}t_{0}}$, $S^* :=S^* + t_{0}$, $T^*:=T^*-t_{0}$ and $f^*(s_{0}):=f(s_{0})-1$.    
It is clear that $\{S^*, T^* \}$ is still a partition of $V$, $B$ is a branching, and  Properties $(a)$, $(b)$ and $(c)$ are still satisfied.  
This \lq suitable\rq~ $t_{0}$ needs to satisfy that after we do the update above, the following inequality still holds:  
\begin{equation} \label{submodular-function-*} 
 \mbox{for any } \emptyset \neq X \subseteq T^*, ~ \widetilde{f^*}(N_{D}^{-}(X)) \geq |X|.  
\end{equation}

The above process continues until $T^* = \emptyset$.  
Then Properties $(a)$ and $(b)$ imply  the final branching $B$ is the one as demanded.   

Now suppose at the induction step of the above process, $T^* \neq \emptyset$. Inequality (\ref{submodular-function-*}) holds by induction hypothesis. 
By Inequality (\ref{submodular-function-*}),  there exists an $s_{0} \in N_{D}^{-}(T^*) \subseteq S^* $ such that $ f^*(s_{0}) >0$. 

Next we show that there exists a $t_{0} \in T^*$ such that $s_{0}t_{0} \in A $ and Inequality (\ref{submodular-function-*}) still holds after we update  $B:=B+ \overrightarrow{s_{0}t_{0}}$, $S^* := S^* + t_{0}$, $T^* :=T^*-t_{0}$ and $f^*(s_{0}) := f^*(s_{0})-1$. 

\smallskip 

\noindent {\bf Case 1.}  Suppose there exists no $ X \subseteq T^*$ such that $ s_{0} \in N_{D}^{-}(X)$ and $\widetilde{f^*}(N_{D}^{-}(X)) = |X| $.  
Then for any $ X \subseteq T^*$ such that $ s_{0} \in N_{D}^{-}(X)$, by Inequality  (\ref{submodular-function-*}), 
we have
\begin{equation}\label{strict>}
\widetilde{f^*}(N_{D}^{-}(X)) > |X| . 
\end{equation} 

Since $s_{0} \in N_{D}^{-}(T^*)$, there exists a $t_{0} \in T^*$ such that $s_{0}t_{0} \in A$. Update  $B:=B+ \overrightarrow{s_{0}t_{0}}$, $S^* :=S^* + t_{0}$, $T^*:=T^*-t_{0}$ and $f^*(s_{0}):=f^*(s_{0})-1$.  

We show that Inequality (\ref{submodular-function-*}) still holds in the induction step. 
Let $X \subseteq T^*$. If $s_{0} \notin N_{D}^{-}(X)$, then by the induction hypothesis of Inequality (\ref{submodular-function-*}), $ \widetilde{f^*}(N_{D}^{-}(X)) \geq |X|$. If $s_{0} \in N_{D}^{-}(X)$, then by (\ref{strict>}), in the induction step, $ \widetilde{f^*}(N_{D}^{-}(X))+1 > |X|$, thus $ \widetilde{f^*}(N_{D}^{-}(X)) \geq |X|$. 

\smallskip

\noindent {\bf Case 2 (not Case 1).}  Suppose there exists some $ X_{0} \subseteq T^*$ such that $ s_{0} \in N_{D}^{-}(X_{0})$ and $\widetilde{f^*}(N_{D}^{-}(X_{0})) = |X_{0}| $. 
Define 
$$\mathcal{E} = \{X \subseteq T^*:  s_{0} \in N_{D}^{-}(X) \text{ and }  \widetilde{f^*}(N_{D}^{-}(X)) = |X| \}.$$  
Next we show two claims.  

\begin{claim} \label{prove-to-be-submodular}
For $X_{1}, X_{2} \subseteq T^*$ such that $s_{0} \in N_{D}^{-}(X_{1}) \cap N_{D}^{-}(X_{2})$,
\begin{equation} \label{submodular-inequality}
\widetilde{f^*}(N_{D}^{-}(X_{1})) +  \widetilde{f^*}(N_{D}^{-}(X_{2})) \geq \widetilde{f^*}(N_{D}^{-}(X_{1} \cup X_{2})) + \widetilde{f^*}(N_{D}^{-}(X_{1} \cap X_{2})).
\end{equation}
Moreover, if the equality holds, then $s_{0} \in N_{D}^{-}(X_{1} \cup X_{2}) \cap N_{D}^{-}(X_{1} \cap X_{2})$.
\end{claim}

\begin{proof}
For $X \subseteq V$, define $\chi_{X}: V \rightarrow \{0, 1 \}$ as 
\[
\chi_{X}(v)=
\begin{cases}
1, & v \in X,\\ 
0, & v \notin X.
\end{cases}
\]
Then $\widetilde{f^*}(X)= \sum_{v \in V} f^*(v)\chi_{X}(v)  $. 
Next we show that for any $v \in V$, 
\begin{equation}\label{chi-submodular}
\chi_{N_{D}^{-}(X_{1})}(v) + \chi_{N_{D}^{-}(X_{2})}(v) \geq \chi_{N_{D}^{-}(X_{1} \cup X_{2})}(v) + \chi_{N_{D}^{-}(X_{1} \cap X_{2})}(v).
\end{equation}
If $\chi_{N_{D}^{-}(X_{1} \cup X_{2})}(v)=1$, that is $v \in N_{D}^{-}(X_{1} \cup X_{2}) $, then $v \in N_{D}^{-}(X_{1})$ or $N_{D}^{-}(X_{2})$; thus $\chi_{N_{D}^{-}(X_{1})}(v)=1$ or $ \chi_{N_{D}^{-}(X_{2})}(v)=1$. 
If $\chi_{N_{D}^{-}(X_{1} \cap X_{2})}(v)=1$, then $v \notin X_{1} \cap X_{2}$; thus $ v \in N_{D}^{-}(X_{1})$ (if $v \notin X_{1}$) or $ v \in N_{D}^{-}(X_{2}) $ (if $v \notin X_{2}$). 
If $\chi_{N_{D}^{-}(X_{1} \cup X_{2})}(v)= \chi_{N_{D}^{-}(X_{1} \cap X_{2})}(v)= 1$, then $v \notin X_{1} \cup X_{2}$; thus $\chi_{N_{D}^{-}(X_{1} \cap X_{2})}(v)= 1$ implies $v \in  N_{D}^{-}(X_{1}) \cap N_{D}^{-}(X_{2})$. 
These three cases prove Inequality (\ref{chi-submodular}).  

It follows from (\ref{chi-submodular}) that 
\begin{equation}\label{eq-chi-submodular}
\begin{split}
& \sum_{v \in V}f^*(v)\chi_{N_{D}^{-}(X_{1})}(v) + \sum_{v \in V}f^*(v)\chi_{N_{D}^{-}(X_{2})}(v) \\
\geq ~~ & \sum_{v \in V} f^*(v)\chi_{N_{D}^{-}(X_{1} \cup X_{2})}(v) + \sum_{v \in V}f^*(v)\chi_{N_{D}^{-}(X_{1} \cap X_{2})}(v).   
\end{split} 
\end{equation}  
This proves $\widetilde{f^*}(N_{D}^{-}(X_{1})) +  \widetilde{f^*}(N_{D}^{-}(X_{2})) \geq \widetilde{f^*}(N_{D}^{-}(X_{1} \cup X_{2})) + \widetilde{f^*}(N_{D}^{-}(X_{1} \cap X_{2}))$. 

By assumption, $s_{0} \in N_{D}^{-}(X_{1}) \cap N_{D}^{-}(X_{2})$ and $f^*(s_{0}) >0$.  If the equality of (\ref{submodular-inequality}) (or equivalently, the equality of (\ref{eq-chi-submodular})) holds, it follows from (\ref{chi-submodular}) that 
\[
\chi_{N_{D}^{-}(X_{1})}(s_{0}) + \chi_{N_{D}^{-}(X_{2})}(s_{0}) = \chi_{N_{D}^{-}(X_{1} \cup X_{2})}(s_{0}) + \chi_{N_{D}^{-}(X_{1} \cap X_{2})}(s_{0}).
\]

Since $s_{0} \in N_{D}^{-}(X_{1}) \cap N_{D}^{-}(X_{2})$, we have $s_{0} \in N_{D}^{-}(X_{1} \cup X_{2}) \cap N_{D}^{-}(X_{1} \cap X_{2})$. 
This proves the claim. 
\end{proof}


\begin{claim}\label{intersection-closure}
If $X_{1}, X_{2} \in \mathcal{E}$, then $X_{1} \cap X_{2} \in \mathcal{E}$.
\end{claim} 

\begin{proof}
	Note that $|X_{1}| + |X_{2}|= |X_{1} \cup X_{2}| + |X_{1} \cap X_{2}|$.
	Combining this with Inequality (\ref{submodular-inequality}) of Claim  \ref{prove-to-be-submodular}, we have
	\begin{equation}\label{submodular-||}
		\begin{split}
			& \widetilde{f^{*}}(N_{D}^{-}(X_{1}))-|X_{1}| +  \widetilde{f^{*}}(N_{D}^{-}(X_{2}))-|X_{2}|\\
			\geq ~ & \widetilde{f^{*}}(N_{D}^{-}(X_{1} \cup X_{2}))-|X_{1} \cup X_{2}| + \widetilde{f^{*}}(N_{D}^{-}(X_{1} \cap X_{2}))-|X_{1} \cap X_{2}|.
		\end{split}
	\end{equation}
	Since $X_{1}, X_{2} \in \mathcal{E}$, we have $\widetilde{f^{*}}(N_{D}^{-}(X_{i}))=|X_{i}|$ ($i=1, 2$).  It follows from  (\ref{submodular-||}) that
	\begin{equation}\label{submodular-=}
		0 \geq \widetilde{f^{*}}(N_{D}^{-}(X_{1} \cup X_{2}))-|X_{1} \cup X_{2}| + \widetilde{f^{*}}(N_{D}^{-}(X_{1} \cap X_{2}))-|X_{1} \cap X_{2}|.
	\end{equation}
	
	By Inequality (\ref{submodular-function-*}), $\widetilde{f^{*}}(N_{D}^{-}(X_{1} \cup X_{2})) \geq |X_{1} \cup X_{2}|$ and  $\widetilde{f^{*}}(N_{D}^{-}(X_{1} \cap X_{2})) \geq |X_{1} \cap X_{2}|$; combining  Inequality (\ref{submodular-=}),
	we have $\widetilde{f^{*}}(N_{D}^{-}(X_{1} \cup X_{2}))=|X_{1} \cup X_{2}|$ and $\widetilde{f^{*}}(N_{D}^{-}(X_{1} \cap X_{2}))=|X_{1} \cap X_{2}|$.
	
	So the equality of (\ref{submodular-||}) holds, therefore the equality of (\ref{submodular-inequality}) holds. Since $X_{1}, X_{2} \in \mathcal{E}$, $s_{0} \in N^{-}_{D}(X_{1}) \cap N^{-}_{D}(X_{2})$. By Claim~\ref{prove-to-be-submodular}, we have $s_{0} \in N^{-}_{D}(X_{1} \cup X_{2}), N_{D}^{-}(X_{1} \cap X_{2})$.
	Hence, $X_{1} \cap X_{2} \in \mathcal{E}$.
\end{proof}


For  $X_{1}, X_{2} \in \mathcal{E}$, define  $X_{1} \prec X_{2}$ if and only if   $X_{1} \subseteq X_{2}$. By Claim \ref{intersection-closure}, we can suppose  $\emptyset \neq X_{0} \in \mathcal{E}$ is a minimal element in $\mathcal{E}$ with respect to $\prec$.   
We pick a $t_0 \in X_0$ with $s_0t_0 \in A$. 
Update $B:=B+ \overrightarrow{s_{0}t_{0}}$, $S^*:=S^* + t_{0}$, $T^*:=T^*-t_{0}$ and $f^*(s_{0}) := f^*(s_{0})-1$. 


We are left to show that Inequality (\ref{submodular-function-*}) still holds after these updates.  


Suppose $\emptyset \neq X \subseteq T^*$. 
If $s_{0} \notin N_{D}^{-}(X)$, then by induction hypothesis of  (\ref{submodular-function-*}),  
$ \widetilde{f^*}(N_{D}^{-}(X)) \geq |X|$.  
Next, suppose $s_{0} \in N_{D}^{-}(X)$, and suppose to the contrary of Inequality (\ref{submodular-function-*}),   
that $\widetilde{f^*}(N_{D}^{-}(X)) < |X|$, then $\widetilde{f^*}(N_{D}^{-}(X))+1 \leq |X|$.  
By induction hypothesis of (\ref{submodular-function-*}), we have  $\widetilde{f^*}(N_{D}^{-}(X))+1 \geq |X|$. 
Thus $\widetilde{f^*}(N_{D}^{-}(X))+1 = |X|$. Combining $s_{0} \in N_{D}^{-}(X)$, we have $X \in \mathcal{E}$ before the updates. Applying  Claim~\ref{intersection-closure}, we have $\emptyset \neq X \cap X_{0} \in \mathcal{E} $. Notice the fact that $t_{0} \in X_{0} \setminus X$. Hence, $X \cap X_{0} \subsetneqq X_{0}$, but this contradicts the minimality of $X_{0}$. 
The contradiction proves $ \widetilde{f^*}(N_{D}^{-}(X)) \geq |X|$.  This proves  \textbf{Case 2}, thus Proposition~\ref{B_{k+1}-exist}.
\end{proof}  

Next we use Proposition~\ref{B_{k+1}-exist} proving the cases  $d \leq k$ of Conjecture \ref{directed-NDT-conj}. 

\begin{thm}\label{main-branching-dec}
	For positive integers $k$ and $d$ with $d \leq k$, if $D=(V,A)$ is a digraph with 
	$\gamma(D) \leq  k + \frac{d-k}{d+1}$ 
	and $\Delta^{-}(D) \leq k+1$, then $D$ decomposes into $k + 1$ branchings $B_{1}, \ldots, B_{k}, B_{k+1}$ with  $\Delta^{+}(B_{k+1}) \leq d$. 
\end{thm}

\begin{proof} 
	Note that since $d \leq k$, we have $\gamma(D) \leq  k + \frac{d-k}{d+1} \leq k$.  
	Suppose $S = \{v \in V: d_{D}^{-}(v) \leq k \}$, $T=\{v \in V: d_{D}^{-}(v)=k+1 \}$,  then $\{S, T \}$ is a partition of $V$. Define $f: V \rightarrow \mathbb{N}$ as  $f(v)=d$ for all $v \in V$. 
	
	
	If $T = \emptyset$, then $\Delta^{-}(D) \leq k$. As $\gamma(D) \leq k$, by Theorem~\ref{directed-nash-williams-theorem}, $D$ decomposes into $k$ branchings.  So we suppose $T \neq \emptyset$. 
	
	For $\emptyset \neq X \subseteq T$, let $D_{X}$ denote the subdigraph of $D$ induced by $X \cup N_{D}^{-}(X)$. Then $d^{-}_{D_{X}}(v)=k+1$ for each $v \in X$, and $|A(D_{X})| \geq \sum_{v \in X}d_{D_{X}}^{-}(v)=(k+1)|X|$.   
	By the assumption $\gamma(D) \leq  k + \frac{d-k}{d+1} = \frac{d(k+1)}{d+1}$, 
	we have
	\[
	\frac{(k+1) |X|}{|X|+ |N_{D}^{-}(X)|} \leq \frac{|A(D_{X})|}{|V(D_{X})|} < \frac{|A(D_{X})|}{|V(D_{X})|-1} \leq \gamma(D) \leq  \frac{d(k+1)}{d+1}.
	\]
	It follows that $d|N_{D}^{-}(X)| > |X|$, that is $\widetilde{f}(N_{D}^{-}(X)) >  |X|$. 
	
	By Proposition~\ref{B_{k+1}-exist}, there exists a branching $B_{k+1}$ such that: $(i)$ for any $v \in T$, $d_{B_{k+1}}^{-}(v)=1$; $(ii)$ for any $v \in V$, $d_{B_{k+1}}^{+}(v)\leq f(v)=d $, that is $\Delta^{+}(B_{k+1}) \leq d$. Since   $\Delta^{-}(D) \leq k+1$ and  $d_{B_{k+1}}^{-}(v)=1$ for $v \in T$,   $\Delta^{-}(D-A(B_{k+1})) \leq k$.
	Since $\gamma(D-A(B_{k+1})) \leq \gamma(D) \leq k$, by Theorem~\ref{directed-nash-williams-theorem}, $D-A(B_{k+1})$ decomposes into $k$ branchings $B_{1}, \ldots, B_{k}$. This proves that $D$ decomposes into $k+1$ branchings $B_{1}, \ldots,  B_{k}, B_{k+1}$ with $\Delta^{+}(B_{k+1}) \leq d$. This proves Theorem~\ref{main-branching-dec}.      
\end{proof}

\section{The pseudo-branching analogue for the NDT Theorem}  

In this section, we prove an analogue in digraphs corresponding to Theorem  \ref{pseudo-forest-NDT}, which is  the following theorem. Note that Theorem  \ref{main-pseudo-branching-dec} is also a  pseudo-branching analogue of Conjecture \ref{directed-NDT-conj} in digraphs.  

\smallskip  

\noindent
{\bf Theorem~\ref{main-pseudo-branching-dec}}.  {\it  
	For positive integers $k$ and $d$, if $D$ is a digraph with $\frac{1}{2}\mad(D)$   $\leq$  $k + \frac{d-k}{d+1}$  
	and $\Delta^{-}(D) \leq k+1$, then $D$ decomposes into $k + 1$ pseudo-branchings $C_{1}, \ldots, C_{k}, C_{k+1}$ with $\Delta^{+}(C_{k+1}) \leq d$.  
}

\smallskip  

\begin{proof} 
	The proof is by contradiction. Suppose $D$ is a counterexample of Theorem~\ref{main-pseudo-branching-dec} with minimal number of arcs.  
	By Theorem \ref{Hakimi}, we suppose $D$ decomposes into $k+1$ pseudo-branchings $C_{1}, \ldots, C_{k}, C_{k+1}$ such that $ \lambda = \sum_{v \in V(D)} \max  \{d^{+}_{C_{k+1}}(v)-d, 0 \}$ is minimum.  Then $\Delta^{+}(C_{k+1}) > d$ and $ \lambda >0 $.   We name $ \lambda $ the residue of this decomposition.

\begin{claim}\label{in-degree=k+1-v}
For vertices $v \in V(D)$, if $d^{-}_{D}(v) > 0$, then $d^{-}_{D}(v)=k+1$.  
\end{claim}

\begin{proof}
	As $\Delta^{-}(D)\leq k+1$, we have $d^{-}_{D}(u) \leq k+1$ for all $u \in V(D)$.  
	Suppose to the contrary that there exists some vertex $v \in V(D)$ such that $ 0< d_{D}^{-}(v) \leq  k$. By the minimality of $A(D)$, $D - \partial_{D}^{-}(v)$ decomposes into  $k+1$ pseudo-branchings $C'_{1}, \ldots, C'_{k}, C'_{k+1}$ such that $\Delta^{+}(C'_{k+1}) \leq d$. Suppose  $\partial_{D}^{-}(v)= \{a_{1}, \ldots, a_{l} \}$, where $1 \le l \leq k$. For $1 \leq i \leq l$,
	$C'_{i}+a_{i}$ is still a pseudo-branching. Hence, $D$ decomposes into $k+1$ pseudo-branchings $C'_{1}+a_{1}, \ldots, C'_{l}+a_{l}$, $C'_{l+1}, \ldots, C'_{k+1}$ such that $\Delta^{+}(C'_{k+1}) \leq d$, a contradiction to the assumption.
\end{proof}

For all $v \in V(D)$, since each $C_i$ ($1\le i \le k+1$) is a  pseudo-branching,  $d_{C_{i}}^{-}(v) \leq 1$.  
Throughout the proof,  we fix a vertex $v_{0}$ such that $d_{C_{k+1}}^{+}(v_{0}) > d$.    
An {\em alternating trail (from $v_0$)} is a vertex-arc sequence $v_{0} \overrightarrow{a_{1}}v_{1} \overleftarrow{a_2}v_2 $ $\ldots $ $\overrightarrow{a_{2t-1}} v_{2t-1} \overleftarrow{a_{2t}} v_{2t} $ 
such that $\overrightarrow{a_{1}} \overleftarrow{a_2} $ $\ldots $  $\overrightarrow{a_{2t-1}}  \overleftarrow{a_{2t}}$ are distinct and alternating in  directions along the trail,  
$\overrightarrow{a_{2i-1}} =  \overrightarrow{v_{2i-2}v_{2i-1}} \in A(C_{k+1})$ 
and $\overleftarrow{a_{2i}} = \overleftarrow{v_{2i-1}v_{2i}} \in A(D)$ for $1 \leq i \leq t$. An alternating trail is allowed to have an odd number of arcs by omitting  the last arc in the above definition.  

Note the following fact: 
For $1 \leq i \leq t$, since $\overrightarrow{a_{2i-1}} = \overrightarrow{v_{2i-2}v_{2i-1}} \in A(C_{k+1})$ and $d_{C_{k+1}}^{-}(v_{2i-1}) = 1$, we derive that $\overleftarrow{a_{2i}} \notin A(C_{k+1})$.  

\begin{claim}\label{v_{2i}-geqk}
If $v_{0} \overrightarrow{a_{1}}v_{1} \overleftarrow{a_2}v_2 $ $\ldots $ $\overrightarrow{a_{2t-1}} v_{2t-1} \overleftarrow{a_{2t}} v_{2t} $ 
is an alternating  trail, then $ d_{C_{k+1}}^{+}(v_{2i}) \geq d$ for $1 \leq i \leq t$. 
\end{claim}

\begin{proof}
Suppose otherwise, and $i_{0}$ is the minimum such that $1 \leq i_{0} \leq  t$  and $ d_{C_{k+1}}^{+}(v_{2i_{0}}) < d$. 
For $1 \leq i \leq i_{0}$, we suppose $\overleftarrow{a_{2i}} \in A(C_{j_{i}})$,  then  $C_{j_{i}} \neq C_{k+1}$.   

Do the following updates in turn: 
$ C_{k+1}: = C_{k+1}- \overrightarrow{a_{1}}+\overleftarrow{a_{2}}$ and $ C_{j_{1}}: = C_{j_{1}}-\overleftarrow{a_{2}}+\overrightarrow{a_{1}}$,  
$\ldots$,  $ C_{k+1}: = C_{k+1}- \overrightarrow{a_{2i_{0}-1}}+\overleftarrow{a_{2i_{0}}}$ and $ C_{j_{i_{0}}}: = C_{j_{i_{0}}}-\overleftarrow{a_{2i_{0}}}+\overrightarrow{a_{2i_{0}-1}}$. 
Observe that, after this series of updates, $C_{1}, \ldots,  C_{k}, C_{k+1}$ are still pseudo-branchings; $ d^{+}_{C_{k+1}}(v_{0})$ is decreased by $1$, $ d^{+}_{C_{k+1}}(v_{2i_{0}})$ is increased by $1$, and $ d^{+}_{C_{k+1}}(v)$ does not  change for $v \neq v_{0}$ or $v_{2i_{0}}$. Hence, $\sum_{v \in V(D)} \max \{d^{+}_{C_{k+1}}(v)-d, 0 \}$ is decreased by $1$, this contradicts that originally we have picked a decomposition with the minimum residue $\lambda$.  
\end{proof}
 
 
Next we define 	
 \[
  	V(v_0) :=  \cup V(Q),  \mbox{ and }  A(v_0) :=  \cup A(Q), \mbox{ where } {{Q} \mbox{ is an alternating trail}   \mbox{ from } v_0}. 
 \]  
Let $ D(v_0)  = (V(v_0), A(v_0))$. The following observation comes from the definitions.   
\begin{observation} \label{coboundary} 
	\begin{itemize}
		\item[(i)] 	$\partial_{C_{k+1}}^{+}(v_0) \subseteq A(v_0)$; 
		
		\item[(ii)] For $v \in V(v_0)$, if $d_{D(v_0)}^{-}(v) > 0$, then $\partial_{D}^{-}(v) \subseteq A(v_0)$;
		
		\item[(iii)] 	For $v \in V(v_0)$, if $d_{D(v_0)}^{+}(v) >0$, then $\partial_{C_{k+1}}^{+}(v) \subseteq A(v_0)$. 
	\end{itemize}
\end{observation}

Let $Z_1 = \{v \in V(v_0): d_{D(v_0)}^{-}(v)>0\}$, $Z_{2}=V(v_0) \setminus Z_{1}$. 
For $v \in Z_{1}$, since $d_{D(v_0)}^{-}(v)>0$, by Observation~\ref{coboundary} $(ii)$, we have $\partial_{D}^{-}(v) \subseteq A(v_0)$; applying  Claim~\ref{in-degree=k+1-v},  we have $d_{D(v_0)}^{-}(v) = d_{D}^{-}(v) = k+1$;
 thus $|A(v_0)| = \sum_{v \in Z_{1}} d_{D(v_0)}^{-}(v)=(k+1) |Z_{1}|$.   
 
For $v \in Z_{2}$, since $d_{D(v_0)}^{-}(v)=0$, we have $d_{D(v_0)}^{+}(v) >0 $.   
By Observation~\ref{coboundary} $(iii)$, $\partial_{C_{k+1}}^{+}(v) \subseteq A(v_0)$. 
Since each $C_{k+1}$ is a  pseudo-branching, for any vertex $u$, $d_{C_{k+1}}^{-}(u) \leq 1$; thus, in $D(v_0)$, each vertex in $Z_{2}$ has at least  $d_{C_{k+1}}^{+}(v)$  
distinct out-neighbors in $Z_{1}$, thus    
 $|Z_{1}| \geq \sum_{v \in Z_{2}}  d_{C_{k+1}}^{+}(v)$.  
Applying Claim~\ref{v_{2i}-geqk}, we have $d_{C_{k+1}}^{+}(v) \geq d$. 

By Observation~\ref{coboundary} $(i)$, $\partial_{C_{k+1}}^{+}(v_{0}) \subseteq A(v_0)$.  Recall that $d_{C_{k+1}}^{+}(v_{0}) > d$. This proves $|Z_{1}| > d|Z_{2}|$ (no matter if $v_0 \in Z_1$ or  $v_0 \in Z_2$). 
It follows that   
\[
\frac{1}{2} \mad(D) \geq \frac{1}{2} \mad(D(v_0)) \geq   \frac{|A(v_0)|}{|V(v_0)|}=\frac{(k+1)|Z_{1}|}{|Z_{1}|+ |Z_{2}|} 
> \frac{(k+1)|Z_{1}|}{|Z_{1}| + \frac{1}{d}|Z_{1}|} =  k + \frac{d-k}{d+1}. 
\]

This contradicts the assumption that $\frac{1}{2} \mad(D) \leq  k + \frac{d-k}{d+1}$, and finishes the proof. 
\end{proof}

\section{Sharpness of the bounds} 
\label{Section-lower-B}

In this section, by constructing a series of acyclic bipartite digraphs, we prove  Theorem~\ref{fra-arb-bound-tight}, which says the bound of  $ \gamma(D)$ given in Conjecture~\ref{directed-NDT-conj} is best possible. The same  digraphs will be used to show the bound of  $ \mad(D)$ in Theorem \ref{main-pseudo-branching-dec}   
is best possible, this shall be explained at the end of this section. 

Suppose $G = (V,E)$ is a graph and $D$ a digraph on the same vertex set $V$. For $X, Y \subseteq V$, 
denote by $[X,Y]_{G}$ the number of edges $uv \in E(G)$ with $u \in X$  and $v \in Y$,  by $[X,Y]_{D}$ the number of arcs $\overrightarrow{uv} \in A(D)$ with $u \in X$ and $v \in Y$. 


 
\smallskip

\noindent
{\bf Theorem~\ref{fra-arb-bound-tight}}.  {\it   
	For positive integers $k, d$ and any $\epsilon > 0$, there exists an acyclic bipartite digraph $D$ such that 
	$ \gamma(D) <  k + \frac{d-k}{d+1} + \epsilon$ 
	and $\Delta^{-}(D) \leq k+1$, but $D$ can not decompose into $k + 1$ branchings $B_{1}, \ldots, B_{k}, B_{k+1}$ with $\Delta^{+}(B_{k+1}) \leq d$.
}


\smallskip

\begin{proof}
	For positive integers $k, d$ and any $\epsilon > 0$, 	
	since $\lim\limits_{n \rightarrow \infty} \limits{ \frac{d(k+1)n}{(d+1)n-1}} =  \frac{d(k+1)}{d+1} < k + \frac{d-k}{d+1} + \epsilon$, there is an integer $n$  such that $ \frac{d(k+1)n}{(d+1)n-1} < k + \frac{d-k}{d+1} + \epsilon$ and $n \geq k+1$. 
	Define   
	 $U = \{u_{0}, \ldots, u_{n-1} \}$; $\mathbb{Z}_{n} = \{0, \ldots, n-1\}$.
	 For convenience,  
	for $t \in \mathbb{Z}$ (or index $t$ in $u_t$),  since $(i+tn)\mod{n} = i$, we  regard $i+tn$ as $i$.

	For each $i \in \mathbb{Z}_{n}$, let $W_{i}$ be a set of $d$ vertices, such that, for distinct $i, j$, the sets $W_{i}$ and $W_{j}$ are disjoint from each other.  
	Define $D_0$ as a bipartite digraph with vertex bipartition $U$ and $\cup_{i \in \mathbb{Z}_{n}} W_i$, and 
	\[
	\begin{split}	
	A(D_0) & = \{\overrightarrow{u_{i+j}w}: u_{i+j} \in U,  w \in W_i, \mbox{ and } 0 \le j \le k \} \\ 
	& = \{\overrightarrow{u_iw}:  u_i \in U,  w \in W_{i-j}, \mbox{ and } 0 \le j \le k \}. 
	\end{split}
	\] 
	
	By taking two disjoint copies $D_{1}, D_{2}$ of $D_{0}$ and identifying two copies of $u_{0}$ into a single vertex $u_{0}^{*}$, we obtain a digraph $D$. We shall show that $D$ satisfies the theorem.   
	It is clear that $\Delta^{-}(D)=k+1$, $D$ is acyclic and bipartite. 
	 	
	\begin{claim}\label{u-out-degree}
		Suppose $D_{0}$ decomposes into $k + 1$ branchings $B'_{1}, \ldots, B'_{k}$,  $B'_{k+1}$ such that $\Delta^{+}(B'_{k+1})$ $\leq$ $d$, then $d_{B'_{k+1}}^{+}(u)=d$ for all $u \in U$. 
	\end{claim} 
	

	\begin{proof}
		For $w \in \cup_{i \in \mathbb{Z}_{n}}W_{i}$,    
		since $k+1=d^{-}_{D_{0}}(w)= \sum_{j=1}^{k+1}d_{B'_{j}}^{-}(w) \le k+1$ 
		(the last inequality comes from  the fact that $B'_{j}$ is a branching, thus $d_{B'_{j}}^{-}(w) \leq 1 $), we deduce that 
		$d_{B'_{j}}^{-}(w) = 1 $ (for  $1 \leq j \leq k+1$).    
		So  
		\[
		\begin{split}
		dn=|\cup_{i \in \mathbb{Z}_{n}}W_{i}|=\sum_{w \in \cup_{i \in \mathbb{Z}_{n}}W_{i}} d_{B'_{k+1}}^{-}(w) & =[U, \cup_{i \in \mathbb{Z}_{n}}W_{i}]_{B'_{k+1}} \\
		& =\sum_{u \in U} d_{B'_{k+1}}^{+}(u) \leq \Delta^{+}(B'_{k+1}) n \leq dn. 
		\end{split} 
		\]
		Therefore  $d_{B'_{k+1}}^{+}(u)=d$ for all $u \in U$. 
	\end{proof}

Assume that $D$ decomposes into $k + 1$ branchings $B_{1}, \ldots, B_{k}, B_{k+1}$ with $\Delta^{+}(B_{k+1}) \leq d$.    
Then $D_{l}$ ($l=1$, $2$) decomposes into branchings $B_{1}[V(D_{l})], \ldots, B_{k+1}[V(D_{l})] $ with $\Delta^{+}(B_{k+1}[V(D_{l})] ) \leq d$.  
By Claim~\ref{u-out-degree}, $d_{B_{k+1}[V(D_{l})]}^{+}(u_{0}^{*})=d$. Thus $d_{B_{k+1}}^{+}(u_{0}^{*})=2d$, this contradicts that $\Delta^{+}(B_{k+1}) \leq d$.  
This proves that  $D$ can not decompose into $k + 1$ branchings $B_{1}, \ldots, B_{k}, B_{k+1}$ with $\Delta^{+}(B_{k+1}) \leq d$.  

In the rest of the proof, it suffices for us to show that $\gamma(D) <  k + \frac{d-k}{d+1} + \epsilon$. 
Note that 
\begin{equation} \label{den-ALL}
\frac{|A(D)|}{|V(D)|-1} = \frac{2dn(k+1)}{(2(dn+n)-1)-1} =  \frac{d(k+1)n}{(d+1)n-1} = \frac{|A(D_0)|}{|V(D_0)|-1}  <  k + \frac{d-k}{d+1} + \epsilon.  
\end{equation}

By the construction and symmetry between $D_1$ and $D_2$ of $D$, we have  $\gamma(D_{0})=\gamma(D)$. To see this, let $\emptyset \neq X \subseteq V(D)$. If $X \subseteq V(D_{1}) $ or $V(D_{2})$, then by definition, $|A[X]| \leq \gamma(D_{0})(|X|-1)$. Otherwise, $|A[X]| = |A[X_{1}]|+|A[X_{2}]|$ and $|X_{1}|+|X_{2}| \leq |X|+1$, where $X_{l}=X \cap V(D_{l}) $ for $l=1, 2$. Then we have $|A[X]| = |A[X_{1}]|+|A[X_{2}]| \leq \gamma(D_{0})(|X_{1}|-1)+\gamma(D_{0})(|X_{2}|-1) \leq \gamma(D_{0})(|X|-1)$.  Hence, if $|X|>1$, then $\frac{|A[X]|}{|X|-1} \leq \gamma(D_{0})$; thus $\gamma(D) \leq \gamma(D_{0})$. Since $D_{0}$ is isomorphic to a subdigraph of $D$, we have $\gamma(D_{0}) \leq \gamma(D)$. So $\gamma(D_{0}) = \gamma(D)$. 

\emph{Suppose $ V^{*} \subseteq V(D_{0})$, $D_{0}[V^{*}] $ witnesses $\gamma(D_{0})$, and subject to this, $V^{*}$ is maximal.}  
In Claim \ref{real-subset-V-C} we shall prove 
that if $ V^{*} \subsetneqq V(D_{0})$, then  $\frac{|A[V^{*}]|}{|V^{*}|-1} \leq  k + \frac{d-k}{d+1}$. 
Since    
$\frac{|A(D_0)|}{|V(D_0)|-1} = \frac{d(k+1)n}{(d+1)n-1}  > \frac{d(k+1)n}{(d+1)n} =  k + \frac{d-k}{d+1}$, we have $ V^{*} = V(D_{0})$.
Then by Inequality (\ref{den-ALL}), this will prove  $\gamma(D) <  k +  \frac{d-k}{d+1} + \epsilon$.   




\begin{claim} \label{W-part-all}
For $i_{0} \in \mathbb{Z}_{n}$, if $V^{*} \cap W_{i_{0}} \neq \emptyset$, then $W_{i_{0}}  \subseteq V^{*}$.   
\end{claim} 

\begin{proof}  
	Since  $D_{0}[V^{*}] $ witnesses $\gamma(D_{0})$, $\frac{|A[V^{*}]|}{|V^{*}|-1} = \gamma(D_{0})$. Then for  $w \in V^{*} \cap W_{i_{0}}$,    we have
	\[
		\frac{\gamma(D_{0})(|V^{*}|-1)-d_{D_{0}}^{-}(w)}{(|V^{*}|-1)-1}= 	\frac{|A[V^{*}]|-d_{D_{0}}^{-}(w)}{(|V^{*}|-1)-1}=\frac{|A[V^{*}-w]|}{|V^{*}-w|-1}  \leq \gamma(D_{0}).
	\] 	
	This gives us $d_{D_{0}}^{-}(w) \geq \gamma(D_{0})$.

	Assume there exists $w' \in W_{i_{0}} \setminus V^{*}$.  Add $w'$ to $D_{0}[V^{*}]$. Note that $d_{D_{0}[V^{*}+w']}^{-}(w')=d_{D_{0}}^{-}(w)\geq \gamma(D_{0})$. Then  	
	$$\gamma(D_{0}[V^{*}+w']) \geq \frac{|A[V^{*}+w']|}{|V^{*}+w'|-1}= \frac{|A[V^{*}]|+d_{D_{0}[V^{*}+w']}^{-}(w') }{(|V^{*}|-1)+1} \geq \gamma(D_{0}). $$ 
	Therefore $D_{0}[V^{*}+w']$ also witnesses $\gamma(D_{0})$, this contradicts  the maximality of $V^{*}$.   
\end{proof} 


%

%


%
%

By Claim \ref{W-part-all}, suppose $V^{*} = U^{*} \cup (\cup_{i \in I}W_{i})$, where $U^{*} \subseteq U$, $I \subseteq \mathbb{Z}_{n}$. Let $D^{*} = D_{0}[V^{*}] $, then $D^*$ is a bipartite graph with vertex bipartition $U^*$ and $\cup_{i \in I} W_i$, 
and $|V(D^{*})|=|U^{*}|+|\cup_{i \in I}W_{i}|=|U^{*}|+d|I|$.     

\begin{claim} \label{real-subset-V-C}
	If $V^{*} \subsetneqq  V(D_{0}) $,  then 
	$\frac{|A(D^{*})|}{|V(D^{*})|-1} = \frac{|A[V^{*}]|}{|V^{*}|-1} \leq \frac{d(k+1)}{d+1} = k + \frac{d-k}{d+1}$.  
\end{claim}	


\begin{proof}  
If $|U^{*}|>|I| $, then $|V(D^{*})| \geq |I|+1 + d|I|  = (d+1)|I|+1 $.   
For any $w \in \cup_{i \in I}W_{i}$,  we have $d_{D^{*}}^{-}(w) \leq d_{D_{0}}^{-}(w) \leq k+1$.   Thus 
\[
\frac{|A(D^{*})|}{|V(D^{*})|-1}= \frac{\sum_{w \in \cup_{i \in I}W_{i} } d_{D^{*}}^{-}(w) }{|V(D^{*})|-1} \leq \frac{d|I|(k+1)}{(d+1)|I|}=\frac{d(k+1)}{d+1}.   
\]

If $|U^{*}|<|I| $, then $|V(D^{*})| \geq |U^{*}|+d(|U^{*}|+1) \geq (d+1)|U^{*}|+1 $.  For any $u \in U^{*}$, we have $d_{D^{*}}^{+}(u) \leq d_{D_{0}}^{+}(u) \leq d(k+1)$. Thus 
\[ 
\frac{|A(D^{*})|}{|V(D^{*})|-1}= \frac{\sum_{u \in U^{*}} d_{D^{*}}^{+}(u)  }{|V(D^{*})|-1} \leq \frac{d(k+1)|U^{*}|}{(d+1)|U^{*}|}=\frac{d(k+1)}{d+1}. 
\]

%


In the rest of proofs, we suppose $|U^{*}|=|I|<n$, then $|V(D^{*})|=(d+1)|I|$.  
Suppose $G_{0}$ is a bipartite graph with vertex bipartition $U$ and  $\mathbb{Z}_{n}$, $E(G_0) = \{iu_{i+j}:  i \in \mathbb{Z}_{n}, \mbox{ and } 0 \le j \le k \}$. Then  $G_{0}$ is $(k+1)$-regular. 

\begin{claim} \label{edges-between-U0-I-c} 
	Suppose $U^{*} \subsetneqq U$, $I \subsetneqq \mathbb{Z}_{n} $, and $|U^{*}|=|I|$. Then $[U \setminus U^{*}, I]_{G_{0}} \geq \frac{k+1}{2}$. 
\end{claim} 


\begin{proof} 
	Suppose $E^{*}$ is the set of edges between $ U^{*}$ and $I$ in $G_0$.   		
	For $0 \leq j \leq k$, let $M_{j} := \{u_{i+j}i \in E(G_{0}) : i \in I\}$, then $M_j$ is a matching of $G_0$, and  the  set of edges between $U$ and $I$ in $G_0$ are the disjoint union of $M_{0}, \ldots, M_{k}$.  
	We prove the claim by contradiction. Assume that $[U \setminus U^{*}, I]_{G_{0}} < \frac{k+1}{2}$. Then less than $\frac{k+1}{2}$ of $\{M_{j} \}_{j=0}^{k}$ contain edges between $U \setminus U^{*}$ and $I$. So more than  
	$\frac{k+1}{2}$ of $\{M_{j} \}_{j=0}^{k}$ are contained in $E^{*}$. 
	
	Suppose there exists $0 \leq j_{0} \leq k$ such that $M_{j_{0}}, M_{j_{0}+1} \subseteq E^{*}$. 
	Since $|U^{*}|=|I|$, $M_{j_{0}} \subseteq E^{*}$ implies $U^{*}=\{u_{i}:i \in  I+j_{0} \}$, and $M_{j_{0}+1} \subseteq E^{*}$ implies $U^{*}=\{u_{i}:i \in I+j_{0}+1 \}$.  Hence, $I+j_{0}=I+j_{0}+1$. Thus $I=\mathbb{Z}_{n}$, but this  contradicts $I \subsetneqq \mathbb{Z}_{n}$.
	     	 
	Therefore we can suppose there exists no $j_{0}$ ($0 \leq j_{0} \leq k$) such that $M_{j_{0}}, M_{j_{0}+1} \subseteq E^{*}$.   
	Since more than $\frac{k+1}{2}$ of $\{M_{j} \}_{j=0}^{k}$ are contained in $E^{*}$, if $k$ is odd, then there exists some $0 \leq j_{0} \leq k$ such that $M_{j_{0}}, M_{j_{0}+1} \subseteq E^{*}$, this is a contradiction. 
	
	For the last case, we have $k$ is even, and exactly $\{M_{2i}: 0 \le i \le \frac{k}{2} \}$ are contained in $E^{*}$.  
	Since $|U^{*}|=|I|$, $M_{0} \subseteq E^{*}$ implies $U^{*}=\{u_{i}:i \in I \}$,  and $M_{2} \subseteq E^{*}$ implies $U^{*}=\{u_{i}:i \in I+2 \}$.   Hence, $I=I+2$. Thus $n$ is even, $I=2\mathbb{Z}_{n}$ or $I=2\mathbb{Z}_{n}+1$, and $U^{*}=\{u_{i}:i \in I \}$. Then $M_{1}$ is contained in the set of edges  between  $U \setminus U^{*}$ and $I$, thus  $[U \setminus U^{*}, I]_{G_{0}}  \geq |M_{1}|= \frac{n}{2} \ge \frac{k+1}{2}$, but this contradicts the assumption $[U \setminus U^{*},  I]_{G_{0}} < \frac{k+1}{2}$.  
\end{proof}

Since $[U^{*}, I]_{G_{0}}+ [U \setminus U^{*}, I]_{G_{0}}= [U, I]_{G_{0}}=(k+1)|I|$, Claim~\ref{edges-between-U0-I-c} implies $[U^{*}, I]_{G_{0}} = (k+1)|I| -[U \setminus U^{*}, I]_{G_{0}} \leq  (k+1)|I|- \frac{k+1}{2}  $. 

To count  $A(D^{*})$, note that for any $u \in U^{*}$ and $i \in I$, each edge $ui$ in $G_{0}$ corresponds to $d$ arcs from $u$ to $W_{i}$ in $D^{*}$. Hence, $|A(D^{*})|= [U^{*}, \cup_{i \in I}W_{i}]_{D}=d[U^{*}, I]_{G_{0}}$. Then 
\[
\frac{|A(D^{*})|}{|V(D^{*})|-1}= \frac{d[U^{*}, I]_{G_{0}}}{(d+1)|I|-1} \leq \frac{d(k+1)|I|-\frac{d(k+1)}{2}}{(d+1)|I|-1} \leq  \frac{d(k+1)|I|-\frac{d(k+1)}{d+1}}{(d+1)|I|-1}= \frac{d(k+1)}{d+1}.
\] 
This proves Claim \ref{real-subset-V-C}. 
\end{proof}
This finishes the proof of Theorem \ref{fra-arb-bound-tight}.   
\end{proof}

Theorem~\ref{fra-arb-bound-tight} shows the bound of Theorem~\ref{main-branching-dec} is sharp. For positive integers $k, d$ and $\epsilon >0$, let $D$ be the digraph as stated in Theorem~\ref{fra-arb-bound-tight}. Then $ \frac{1}{2}\mad(D) \leq \gamma(D) <  k + \frac{d-k}{d+1} + \epsilon$.  
Since $D$ is acyclic, a subdigraph of $D$ is a branching if and only if it is a pseudo-branching. Hence, $D$ can not decompose into $k+1$ pseudo-branchings $C_{1}, \ldots, C_{k}, C_{k+1}$ with $\Delta^{+}(C_{k+1}) \leq d$. This shows the bound of Theorem~\ref{main-pseudo-branching-dec} is sharp.

\bigskip

\noindent {\bf Acknowledgements:} We thank Professor Xuding Zhu for recommending this research topic to the authors.  We thank two anonymous referees for their detailed and constructive suggestions that help improve the presentation of this paper.

\end{document}